\documentclass[11pt]{article}%{ctexart}
\usepackage{graphicx}
\usepackage[margin=0.5in]{geometry}
\usepackage{bm}
\usepackage{amsthm}
\usepackage{amscd}
\usepackage{amsbsy}
\usepackage{amsmath}
\usepackage{amsfonts}
\usepackage{amssymb}
\usepackage{calligra}
\usepackage{xcolor}
\usepackage{float}
\usepackage[T1]{fontenc}
\usepackage{multirow}
\usepackage{mathrsfs}
\usepackage{mathtools}
\usepackage{epstopdf}
\usepackage{psfrag}
\usepackage{subfigure}
\usepackage{booktabs}
\usepackage{listings}
\usepackage{longtable}
\usepackage{latexsym}
\usepackage{arydshln}
\usepackage{enumerate}
\usepackage{epsfig}
\usepackage{cite}
\usepackage{verbatim}
\usepackage{overpic}
\usepackage{abstract}
\allowdisplaybreaks[4]

\date{}
\title{Sufficient conditions for the $n$-dimensional real Jacobian conjecture}
\author{Changjian Liu$^{\text{a}}$ and Yuzhou Tian$^{\text{b},}$\footnote{Corresponding author: Yuzhou Tian.}\\
	\it\footnotesize School of Mathematics (Zhuhai), Sun Yat-sen University, Zhuhai  519082, China$^{\text{a}}$,\\
	\it\footnotesize E-mail: liuchangj@mail.sysu.edu.cn\\
	\it\footnotesize Department of Mathematics, Jinan University, Guangzhou 510632, China$^{\text{b}}$\\
	\it\footnotesize E-mail: tianyuzhou2016@163.com}

\newtheorem {theorem*}{Theorem}
\newtheorem {theorem} {Theorem}
\newtheorem{thm*}{Theorem}
\newtheorem{thm}{Theorem}

\newtheorem{example}{Example}

\newtheorem{proposition}{Proposition}
\newtheorem{lemma}{Lemma}
\newtheorem{corollary}{Corollary}
\newtheorem{remark}{Remark}

\newtheorem {open problem} {Open problem}
\newtheorem {problem} {Problem}

\numberwithin{equation}{section}

\usepackage[colorlinks,
            CJKbookmarks=true,
            linkcolor=blue,
            anchorcolor=red,
             urlcolor=blue,
            citecolor=blue
            ]{hyperref}

\begin{document}
\maketitle
\noindent {\bf Abstract} The real Jacobian conjecture was posed by Randall in 1983. This conjecture asserts that if $F=\left(f_1,\ldots ,f_n\right):\mathbb{R}^n\rightarrow\mathbb{R}^n$ is a polynomial map  such that  $\det DF\left(\mathbf{x}\right)\neq0$ for all $\mathbf{x}\in\mathbb{R}^n$, then $F$ is injective.

This investigation mainly consists of two parts.  Firstly, we use the qualitative theory of dynamical systems  to give an alternate proof of the polynomial version of the $n$-dimensional Hadamard's theorem. Secondly, we  present some algebraic sufficient conditions for the $n$-dimensional real Jacobian conjecture.  Our results not only extend the main result of [J. Differential Equations {\bf 260} (2016), 5250-5258] to quasi-homogeneous type, but also  generalize it from $\mathbb{R}^2$ to $\mathbb{R}^n$. As a coproduct of our proof process, we solve an open problem formulated
by Braun, Gin\'{e} and Llibre in [J. Differential Equations {\bf 260} (2016), 5250-5258].%the mentioned paper.

\smallskip

\noindent {\bf 2020 Math Subject Classification: } Primary 34C05.  Secondary 34D45.  Tertiary 14R15

\smallskip

\noindent {\bf Key words and phrases:} {Real Jacobian conjecture; Global injectivity; Quasi-homogeneous}

\section{Introduction and main results}\label{se1}
The classical inverse function theorem  tells us that the following  smooth map
\begin{align}\label{eq48}
&F=\left(f_1,\ldots,f_n\right):\mathbb{R}^n\rightarrow\mathbb{R}^n
\end{align}
is a local diffeomorphism with $\det DF\left(\mathbf{x}\right)\neq0$ for all $\mathbf{x}=\left(x_1,\ldots,x_n\right)\in\mathbb{R}^n$, but it is not necessarily globally injective in $\mathbb{R}^n$. This creates an interesting problem to look for some conditions to ensure that $F$ is a globally injective in $\mathbb{R}^n$.  Many excellent  works have been dedicated to this
 topic, such as \cite{MR3319979,MR305418,MR1940235,MR2096702,MR694845}.  The polynomial map $F$ as specific smooth map has attracted the attention of numerous scholars.  For polynomial map, the following two conjectures in algebraic geometry are well-known.

In 1939, Keller \cite{MR1550818} investigated the polynomial map $F$ in \emph{complex space $\mathbb{C}^n$}  and posed a conjecture:  if $F=\left(f_1,\ldots ,f_n\right):\mathbb{C}^n\rightarrow\mathbb{C}^n$ is a polynomial map  such that  $\det DF\left(\mathbf{x}\right)$ is a \emph{non-zero constant} for all  $\mathbf{x}\in\mathbb{C}^n$, then $F$ is injective. This is the famous \emph{Jacobian conjecture}.  In spite of the multiple attempts to prove for many mathematicians, Jacobian conjecture  is still open even for two-dimensional ($n=2$).  In 1998, the Jacobian conjecture was picked by Smale \cite{MR1631413} as one of the 18 great mathematical problems for the 21th century. The recent progress of Jacobian conjecture  can be found in the books \cite{MR1790619,MR4242820} and survey \cite{MR663785}, etc.

In 1983, Randall \cite{MR713265}  turned his attention to the polynomial map $F$ in \emph{Euclidean space}  $\mathbb{R}^n$ and stated the following conjecture: if $F=\left(f_1,\ldots ,f_n\right):\mathbb{R}^n\rightarrow\mathbb{R}^n$ is a polynomial map  such that  $\det DF\left(\mathbf{x}\right)\neq0$ for all $\mathbf{x}\in\mathbb{R}^n$, then $F$ is injective.  This conjecture was named \emph{real Jacobian conjecture}. There have profound relations between the Jacobian conjecture and real Jacobian conjecture: if $F=\left(f_1,\ldots ,f_n\right):\mathbb{C}^n\rightarrow\mathbb{C}^n$ is a polynomial map  such that  $\det DF\left(\mathbf{x}\right)$ is a non-zero constant for all  $\mathbf{x}\in\mathbb{C}^n$, then $\widetilde{F}=\left(\text{Re}f_1,\text{Im}f_1,\ldots,\text{Re} f_n,\text{Im}f_n\right):\mathbb{R}^{2n}\rightarrow \mathbb{R}^{2n}$ is a polynomial map and $\det D\widetilde{F}\left(\mathbf{x}\right)=\big|\det DF\left(\mathbf{x}\right)\big|^2$ is  a non-zero constant, see \cite{MR1790619}. This implies that \emph{if the $2n$-dimensional real Jacobian conjecture is correct, then
the $n$-dimensional Jacobian conjecture holds.} Regrettably, Pinchuk \cite{MR1292168} found  a counterexample with nonvanishing Jacobian determinant (but not a non-zero constant) to the two-dimensional  real Jacobian conjecture. Nevertheless, Pinchuck's  counterexample did not lead
to the mathematicians lost interest in the real Jacobian conjecture.  The mathematical community proposed a fascinating problem, that is, to search additional conditions such that the real Jacobian conjecture is true.  The exploration of this problem has enriched results about polynomial automorphisms  and provides  lots of information available on the Jacobian conjecture.

From now on, we focus on the real Jacobian conjecture. The two-dimensional  real Jacobian conjecture (i.e., $F=\left(f,g\right)$) has been studied
 extensively, and obtained plentiful results. Gwo\'{z}dziewicz proved that if $\text{deg}\;f\leq3$ and $\text{deg}\;g\leq3$, then the two-dimensional  real Jacobian conjecture is correct, see \cite{MR1839866}.  It is shown in \cite{MR2552779,MR3514314}  that the two-dimensional  real Jacobian conjecture holds if $\text{deg}\;f\leq4$, independent of $\text{deg}\;g$. Sabatini in \cite{MR1636592} constructed a global dynamical condition, which allows us to investigate the two-dimensional  real Jacobian conjecture via some dynamical systems  tools. After Sabatini's work, a stream of sufficient conditions are proved. For example,  Braun et al. \cite{MR3448779}  provided a nice sufficient condition as follows.
 \begin{thm}\label{th7}{\rm (see \cite{MR3448779})}
 	Assume that the polynomial map $F=\left(f,g\right):\mathbb{R}^2\rightarrow \mathbb{R}^2$ satisfies  $F\left(0,0\right)=\left(0,0\right)$ and $\det DF\left(x,y\right)\neq0$ for all $\left(x,y\right)\in\mathbb{R}^2$. If the higher homogeneous terms of the polynomials  $ff_x+gg_x$ and $ff_y+gg_y$ do not have real linear factors in common, then $F$ is injective.
 	%If  the higher homogeneous part of $\partial_{\left(x,y\right)}H$ only vanishes at the origin in $\mathbb{R}^2$, then $F$ is injective.
 \end{thm}
\noindent Recently,  many authors either  improve or generalize Theorem \ref{th7}, but thier results only restrict to two-dimensional, see for instance \cite{MR4776971,vall2021,MR4219338,vall2022,MR4213674}.
 For other works of the  two-dimensional  real Jacobian conjecture, see \cite{MR3443401,MR3980031,MR3866202,tian2021necessary}.

However, to the best of our knowledge,  there are only few results about the $n$-dimensional real Jacobian conjecture. The polynomial version of  Hadamard's theorem 
 in $\mathbb{R}^n$ gives a
 necessary and sufficient condition for the $n$-dimensional real Jacobian conjecture, that is,
\begin{thm}\label{th3}{\rm (see \cite{MR3319979,MR787404})}
	Assume that the polynomial map $F=\left(f_1,\ldots,f_n\right):\mathbb{R}^n\rightarrow \mathbb{R}^n$ satisfies  $F\left(\mathbf{0}\right)=\mathbf{0}$ and $\det DF\left(\mathbf{x}\right)\neq0$ for all $\mathbf{x}\in\mathbb{R}^n$.  Define  a norm   $\parallel\cdot\parallel$ in Euclidean space  $\mathbb{R}^n$. Then $F$ is injective if and only if
	\begin{align}\label{eq-47}
		&\lim_{\parallel \mathbf{x}\parallel\rightarrow\infty}\parallel F\left(\mathbf{x}\right)\parallel=\infty.
	\end{align}
\end{thm}

In this paper,  we first provide an alternate proof of Theorem \ref{th3}, which is borrowed from the techniques for qualitative theory of dynamical systems rather than the nonlinear functional analysis, see Section \ref{se4}.

{\bf Convention}. Since the norms $\parallel\cdot\parallel$ are all equivalent in Euclidean space  $\mathbb{R}^n$, unless otherwise specified, we will take the Euclid norm in our proofs, that is,
$$\parallel\mathbf{x}\parallel_2\triangleq\left(x_1^2+\cdots+x_n^2\right)^{\frac{1}{2}}.$$

  By the structure of polynomial map $F$, Cima et al. in \cite{MR1362759,cima1995global}  presented some sufficient  conditions  for the $n$-dimensional real Jacobian conjecture. To introduce the main result of \cite{MR1362759}, we next recall the definition of quasi-homogeneous.

Let $\mathbf{s}=\left(s_1,\ldots,s_n\right)$, $\mathbf{x}=\left(x_1,\ldots,x_n\right)$ and  $\lambda^{\mathbf{s}}\mathbf{x}=\left(\lambda^{s_1}x_1,\ldots,\lambda^{s_n}x_n\right)$.
The polynomial  $\mathcal{P}\left(\mathbf{x}\right)$ of $n$ variables is said to be
\emph{quasi-homogeneous of weighted degree $\ell$ with respect to weight exponents $\mathbf{s}=\left(s_1,\ldots,s_n\right)$} if there exist  vector $\mathbf{s}=\left(s_1,\ldots,s_n\right)\in\mathbb{N}_+^n$ and $\ell\in\mathbb{N}_+$ such that $\mathcal{P}\left(\lambda^{\mathbf{s}}\mathbf{x}\right)=\lambda^\ell\mathcal{P}\left(\mathbf{x}\right)$ for arbitrary $\lambda>0$.  For given a weight exponents $\mathbf{s}=\left(s_1,\ldots,s_n\right)\in\mathbb{N}_+^n$, the polynomial $\mathcal{P}\left(\mathbf{x}\right)$ can be represented in the sum of its $\mathbf{s}$-quasi-homogeneous terms
$\mathcal{P}\left(\mathbf{x}\right)=\mathcal{P}_1\left(\mathbf{x}\right)+\cdots+\mathcal{P}_d\left(\mathbf{x}\right)$, where  $\mathcal{P}_i$ is quasi-homogeneous polynomial of weighted degree $i$ with respect to weight exponents $\mathbf{s}$. In this formula,  the polynomial $\mathcal{P}_d\left(\mathbf{x}\right)$ is called \emph{the higher $\mathbf{s}$-quasi-homogeneous term of $\mathcal{P}\left(\mathbf{x}\right)$}, and the \emph{$\mathbf{s}$-weighted degree of $\mathcal{P}\left(\mathbf{x}\right)$} is defined as $\deg_{\mathbf{s}}\mathcal{P}\left(\mathbf{x}\right)=d$.

Let $f_i^\mathbf{s}$ be the higher $\mathbf{s}$-quasi-homogeneous term of polynomial $f_i$  with $i=1,\ldots,n$. \emph{The higher $\mathbf{s}$-quasi-homogeneous  part of the  polynomial map $F$} is defined by $F_\mathbf{s}=\left(f_1^\mathbf{s},\ldots,f_n^\mathbf{s}\right)$. The main result of \cite{MR1362759} was presented
as follows.
\begin{thm}\label{th5}{\rm (see \cite{MR1362759})}
	Assume that the polynomial map $F=\left(f_1,\ldots,f_n\right):\mathbb{R}^n\rightarrow \mathbb{R}^n$ satisfies  $F\left(\mathbf{0}\right)=\mathbf{0}$ and $\det DF\left(\mathbf{x}\right)\neq0$ for all $\mathbf{x}\in\mathbb{R}^n$. If there exists a weight  exponents $\mathbf{s}=\left(s_1,\ldots,s_n\right)\in\mathbb{N}_+^n$  such that $F_\mathbf{s}$ only vanishes at the origin in $\mathbb{R}^n$, then $F$ is injective.
\end{thm}

%In this paper we will always denote the Bendixson compactification of vector field $X$ by $b\left(X\right)$, see Section \ref{se3} for the definition of Bendixson compactification.

%\begin{remark}\label{re3}
%\emph{Our result not only gives dynamical conditions for $n$-dimensional, but also it is based on gradient vector field and is completely different from Sabatini's theorem, that is, Theorem \ref{sa}. The above conditions (ii) and (vi) are global, and (iii) is local.}
%\end{remark}
	
Let $P_i^{\mathbf{s}}\left(\mathbf{x}\right)$ be the higher $\mathbf{s}$-quasi-homogeneous term of polynomial $P_i\left(\mathbf{x}\right)$  with $i=1,\ldots,n$.  \emph{The higher $\mathbf{s}$-quasi-homogeneous  part of the  polynomial vector field $X=\left(P_1\left(\mathbf{x}\right),\ldots,P_n\left(\mathbf{x}\right)\right)$} is given by $X_{\mathbf{s}}=\left(P_1^{\mathbf{s}}\left(\mathbf{x}\right),\ldots,P_n^{\mathbf{s}}\left(\mathbf{x}\right)\right)$. We denote $\partial_\mathbf{x}\mathcal{P}=\left(\partial_{x_1}\mathcal{P},\ldots,\partial_{x_n}\mathcal{P}\right)$.

In Section \ref{se2}, we present the following two algebraic sufficient conditions to guarantee that the $n$-dimensional real Jacobian conjecture holds.

The first one is similar to Theorem  \ref{th5}, the only difference is that we take the function $\partial_{\mathbf{x}}H_{\mathbf{s}}\left(\mathbf{x}\right)$ instead of  the function
$F_\mathbf{s}$, where $H\left(\bf{x}\right)=\parallel F\left(\bf{x}\right)\parallel_2^2/2$.
\begin{theorem}\label{th10}
	Assume that the polynomial map $F=\left(f_1,\ldots,f_n\right):\mathbb{R}^n\rightarrow \mathbb{R}^n$ satisfies  $F\left(\mathbf{0}\right)=\mathbf{0}$ and $\det DF\left(\mathbf{x}\right)\neq0$ for all $\mathbf{x}\in\mathbb{R}^n$. Let $H\left(\bf{x}\right)=\parallel F\left(\bf{x}\right)\parallel_2^2/2$ and $H_{\mathbf{s}}\left(\mathbf{x}\right)$ be the higher $\mathbf{s}$-quasi-homogeneous term of $H\left(\mathbf{x}\right)$ with  $\mathbf{s}=\left(s_1,\ldots,s_n\right)\in\mathbb{N}_+^n$. If there exists a weight  exponents $\mathbf{s}=\left(s_1,\ldots,s_n\right)\in\mathbb{N}_+^n$  such that $\partial_{\mathbf{x}}H_{\mathbf{s}}\left(\mathbf{x}\right)$ only vanishes at the origin in $\mathbb{R}^n$ (or $\mathbf{x}=\mathbf{0}$ is an isolated  (or a unique) zero of $H_{\mathbf{s}}\left(\mathbf{x}\right)$), then $F$ is injective.
\end{theorem}

It is worth pointing out that the  condition $\partial_{\mathbf{x}}H_{\mathbf{s}}\left(\mathbf{x}\right)$ only vanishes at the origin  and
the condition $\mathbf{x}=\mathbf{0}$ is an isolated  (or a unique) zero of $H_{\mathbf{s}}\left(\mathbf{x}\right)$ are equivalent, this will be
proved in Corollary \ref{co2}.

The second one generalizes the Theorem \ref{th7} not only to quasi-homogeneous type but also to $\mathbb{R}^n$.
\begin{theorem}\label{th2}
	Assume that the polynomial map $F=\left(f_1,\ldots,f_n\right):\mathbb{R}^n\rightarrow \mathbb{R}^n$ satisfies  $F\left(\mathbf{0}\right)=\mathbf{0}$ and $\det DF\left(\mathbf{x}\right)\neq0$ for all $\mathbf{x}\in\mathbb{R}^n$.  Let $H\left(\bf{x}\right)=\parallel F\left(\bf{x}\right)\parallel_2^2/2$. Define a gradient vector field  as follows:
	\begin{align}\label{eq14}
		&\mathscr{Y}=\left(-\partial_{x_1}H,\cdots,-\partial_{x_n}H\right).
	\end{align}
	If there exists a weight  exponents $\mathbf{s}=\left(s_1,\ldots,s_n\right)\in\mathbb{N}_+^n$  such that  the higher $\mathbf{s}$-quasi-homogeneous part of vector field $\mathscr{Y}$ only vanishes at the origin in $\mathbb{R}^n$, then $F$ is  injective.
\end{theorem}

In  \cite{MR3448779}, the authors tried to discuss a possible equivalence between Theorem \ref{th7} and Theorem \ref{th5}  in $\mathbb{R}^2$, and  raised the following open problem. If this problem has positive answer, then Theorem \ref{th5}  implies Theorem \ref{th7} in $\mathbb{R}^2$. However, they only solved this open problem for some special cases.

\vspace{2ex}
\noindent{\bf Open problem.} (see \cite{MR3448779}) \emph{Assume that the polynomial map $F=\left(f,g\right):\mathbb{R}^2\rightarrow \mathbb{R}^2$ satisfies  $F\left(0,0\right)=\left(0,0\right)$ and $\det DF\left(x,y\right)\neq0$ for all $\left(x,y\right)\in\mathbb{R}^2$. If the higher homogeneous part of vector field $\left(ff_x+gg_x,ff_y+gg_y\right)$ only vanishes at the origin in $\mathbb{R}^2$, then is there a weight  exponents $\mathbf{s}=\left(s_1,s_2\right)\in\mathbb{N}_+^2$ such that the higher $\mathbf{s}$-quasi-homogeneous  part of the  polynomial map $F$ only vanishes at the origin in $\mathbb{R}^2$?}
\vspace{2ex}

In the following theorem, we  show that Theorem \ref{th5}  implies  Theorem \ref{th2} in $\mathbb{R}^n$. Specially when $n=2$, it  give a positive answer to the above open problem.

\begin{theorem}\label{th13}
Assume that the polynomial map $F=\left(f_1,\ldots,f_n\right):\mathbb{R}^n\rightarrow \mathbb{R}^n$ satisfies  $F\left(\mathbf{0}\right)=\mathbf{0}$ and $\det DF\left(\mathbf{x}\right)\neq0$ for all $\mathbf{x}\in\mathbb{R}^n$.  Let $H\left(\bf{x}\right)=\parallel F\left(\bf{x}\right)\parallel_2^2/2$. If there exists a weight  exponents $\mathbf{s}=\left(s_1,\ldots,s_n\right)\in\mathbb{N}_+^n$  such that  the higher $\mathbf{s}$-quasi-homogeneous part of vector field $\mathscr{Y}$ given in \eqref{eq14} only vanishes at the origin in $\mathbb{R}^n$, then there exists a weight  exponents $\mathbf{\tilde{s}}=\left(\tilde{s}_1,\ldots,\tilde{s}_n\right)\in\mathbb{N}_+^n$ such that  the higher $\mathbf{\tilde{s}}$-quasi-homogeneous  part of the  polynomial map $F$ only vanishes at the origin in $\mathbb{R}^n$.
\end{theorem}

Inspired by the above open problem, it is natural to ask the next problem.
\begin{problem}\label{p1}
What is the relationship between Theorem \ref{th5}, Theorem \ref{th10} and Theorem \ref{th2}?
\end{problem}

The following result and example will answer  the Problem \ref{p1}.
\begin{theorem}\label{th11}
	Theorem \ref{th10} and Theorem \ref{th2} are equivalent.
\end{theorem}

\begin{example}\label{ex1}{\rm (see \cite{MR3448779})}
\emph{Consider the  polynomial map $F=(f,g)=\left(x^3+y^3+x,y\right)$ with $\det DF=3x^2+1$.  The higher homogeneous  part of $F$ is $\left(x^3+y^3,y\right)$, which only vanishes at the origin. The higher $\mathbf{s}$-quasi-homogeneous term of
$$H=\frac{f^2+g^2}{2}=\frac{x^2}{2}+x^4+\frac{x^6}{2}+\frac{y^2}{2}+x y^3+x^3y^3+\frac{y^6}{2}$$
is $y^6/2$, $(x^3+y^3)^2/2$ and $x^6/2$ depending on $s_1<s_2$, $s_1=s_2$ and $s_1>s_2$, respectively. None of them  only vanishe at the origin in $\mathbb{R}^2$.}
\end{example}
\noindent Example \ref{ex1} shows that  Theorem \ref{th10} or Theorem \ref{th2} does not imply Theorem \ref{th5}. The proof of Theorems \ref{th13} and \ref{th11} will be given in Section \ref{se6}.
%The homogeneous polynomials $\mathcal{P}_1\left(x,y\right)$ and $\mathcal{P}_2\left(x,y\right)$ do not have real linear factors in common if and only if $\left(0,0\right)$ is an unique zero of $\mathcal{P}_1\left(x,y\right)=\mathcal{P}_2\left(x,y\right)=0$.  This shows that

%The condition (ii) of Theorem \ref{th9}  that the higher homogeneous term of $H\left(x,y\right)$ does not have a factor  $\left(ax+by\right)^2$ with $ab\neq0$ is not necessary, and it can be generalized as follows.

 \section{A dynamical proof of Theorem \ref{th3}}\label{se4}
 The main  aim of this section is to give an alternate proof of Theorem \ref{th3} via the method of the theory of dynamical systems.

Let $X=\left(P_1\left(\mathbf{x}\right),\ldots,P_n\left(\mathbf{x}\right)\right)$ be a polynomial vector field of degree $d$. The dynamical behavior of the flows of vector field $X$ near infinity can be characterized by \emph{Poincar\'{e} compactification}. Briefly speaking, the \emph{Poincar\'{e} compactification of vector field $X$} means that it can be extended to an analytic vector field $p\left(X\right)$ in the closed
unit ball $\mathbb{D}^n=\left\{\mathbf{x}\;|\parallel \mathbf{x}\parallel_2\leq1,\;\mathbf{x}\in\mathbb{R}^n\right\}$, and the flow of $p\left(X\right)$ is analytically conjugate to  the flow of $X$ in the interior of $\mathbb{D}^n$.  The boundary of $\mathbb{D}^n$ is a unit sphere $\mathbb{S}^{n-1}=\left\{\mathbf{x}\;|\parallel \mathbf{x}\parallel_2=1,\;\mathbf{x}\in\mathbb{R}^n\right\}$. Denote the restriction of  $p\left(X\right)$ to the unit sphere $\mathbb{S}^{n-1}$ by $p\left(X\right)|_{\mathbb{S}^{n-1}}$.   It is known that the dynamical behavior of the flows of vector field $X$ near infinity  are given by the vector field $p\left(X\right)|_{\mathbb{S}^{n-1}}$. For more details about Poincar\'{e}  compactification in $\mathbb{R}^n$, we refer the reader to \cite{MR998352}. Throughout the rest of this paper, all mentioned the infinity of a vector field is refer to the boundary $\mathbb{S}^{n-1}$ of $\mathbb{D}^{n}$.

Let $q\in\mathbb{R}^n$ be an isolated singular point of  the vector field $X$, and $U_q$ be a neighborhood  of $q$ such that  $X\left(\mathbf{x}\right)\neq\mathbf{0}$ for all $\mathbf{x}\in U_q\setminus\left\{q\right\}$. Consider the following map
 \begin{align*}
 \xi:\partial U_q&\rightarrow \quad\mathbb{S}^{n-1}\\
 \mathbf{x}&\mapsto\dfrac{X\left(\mathbf{x}\right)}{\parallel X\left(\mathbf{x}\right)\parallel}.
 \end{align*}
The \emph{index} of $X$ at singular point  $q$ is defined by the degree of the map $\xi$ and it is  denoted by $i_{X}\left(q\right)$, see \cite{MR0226651} for more details.

 The following result was proved in \cite{MR4066015}.
\begin{theorem}\label{th4}{\rm (see  \cite{MR4066015})}
	Assume that the polynomial vector field $X$  in $\mathbb{R}^n$ only  has finitely many finite singular points $q_1,\ldots,q_l$.  If the infinity of $X$ (i.e. the boundary $\mathbb{S}^{n-1}$ of $\mathbb{D}^n$)  is collapsed to a point, which is a repeller, then
	$$\sum_{j=1}^li_{X}\left(q_j\right)=\left(-1\right)^n.$$
\end{theorem}

The next result can be found in Lemma 4 of page 37 of book \cite{MR0226651}.
\begin{lemma}\label{le3}{\rm (see  \cite{MR0226651})}
Let $q\in\mathbb{R}^n$ be a  singular point  of polynomial vector field $X$ with $\det DX\left(q\right)\neq0$.  Then the index of $X$ at singular point  $q$ is $i_{X}\left(q\right)=\emph{sign}\;\det DX\left(q\right)$.
\end{lemma}

\begin{lemma}\label{le2}
Assume that the polynomial map $F=\left(f_1,\ldots,f_n\right):\mathbb{R}^n\rightarrow \mathbb{R}^n$ satisfies  $F\left(\mathbf{0}\right)=\mathbf{0}$ and $\det DF\left(\mathbf{x}\right)\neq0$ for all $\mathbf{x}\in\mathbb{R}^n$.  Let $H\left(\mathbf{x}\right)=\parallel F\left(\mathbf{x}\right)\parallel_2^2/2$. Then $q\in\mathbb{R}^n$ is a singular point of the gradient vector field defined by
\begin{align}\label{eq9}
&\mathscr{Y}=\left(-\partial_{x_1}H,\cdots,-\partial_{x_n}H\right)
\end{align}
if and only if $F\left(q\right)=\mathbf{0}$. Moreover, the index of $\mathscr{Y}$ at singular point  $q$ is $\left(-1\right)^n$.
\end{lemma}

\begin{proof}
The gradient vector field $\mathscr{Y}$ can be expressed as
\begin{align}\label{eq7}
	& -\left(
	\begin{array}{c}
	\partial_{x_1}H\\
		\vdots\\
		\partial_{x_n}H\\
	\end{array}
	\right)=
	-DF
	\left(
	\begin{array}{c}
		f_1\\
		\vdots\\
		f_n\\
	\end{array}
	\right).
\end{align}
Therefore, $q$ is a singular point of $\mathscr{Y}$  if and only if  it satisfies
\begin{align}\label{eq8}
	&
	DF\left(q\right)
	\left(
	\begin{array}{c}
		f_1\left(q\right)\\
		\vdots\\
		f_n\left(q\right)\\
	\end{array}
	\right)
	=
	\left(
	\begin{array}{c}
		0\\
		\vdots\\
		0\\
	\end{array}
	\right).
\end{align}
Equation \eqref{eq8}  can be regard as a linear system with unknowns $f_1\left(q\right),\ldots,f_n\left(q\right)$.  Since $\text{det} D F\left(q\right)\neq0$,
the  coefficient matrix  of linear system \eqref{eq8}  is  invertible.   So, $q\in\mathbb{R}^n$ is a singular point of $\mathscr{Y}$  if and only if  $F\left(q\right)=\mathbf{0}$.

The determinant of the  matrix $D\mathscr{Y}$ at singular point  $q$ is given by
$$\det D\mathscr{Y} \left(q\right)=\left(-1\right)^n\left(\det DF\left(q\right)\right)^2.$$
Applying Lemma \ref{le3},  the index of $\mathscr{Y}$ at singular point  $q$ is $\left(-1\right)^n$ due to $\det DF\left(q\right)\neq0$ and this lemma  is proved.
\end{proof}

For every polynomial $V\left(\mathbf{x}\right)\in\mathbb{R}\left[\mathbf{x}\right]$, its associated gradient vector field is
\begin{align}\label{eq55}
	&X=\left(-\partial_{x_1}V,\ldots,-\partial_{x_n}V\right).
\end{align}

The asymptotic behavior of flows of gradient vector field \eqref{eq55} are characterized in the  following proposition,  see page 206 of \cite{MR2144536}.
\begin{proposition}\label{pr2}{\rm (see  \cite{MR2144536})}
	Assume that the gradient vector field \eqref{eq55}  only  has isolated finite singular points.  Let the flow $\phi\left(t,p\right)$ generated by vector field \eqref{eq55}  with $\phi\left(0,p\right)=p\in\mathbb{R}^n$. Then the flow $\phi\left(t,p\right)$ either tends to the infinity of $\mathbb{R}^n$ (i.e. the boundary $\mathbb{S}^{n-1}$ of $\mathbb{D}^{n}$), or a singular point of vector field \eqref{eq55} as $t\rightarrow\pm\infty$.
\end{proposition}
%\begin{theorem}\label{th12}
%Let $q$ be an $\omega$-limit point of a flow of gradient vector field \eqref{eq55}. Then $q$ is a singular point of vector field \eqref{eq55}.
%\end{theorem}
%The next corollary follows immediately from the above theorem.
%\begin{corollary}\label{co1}
%
%\end{corollary}
%
%\begin{proof}
%This corollary is trivial when  $\phi\left(t,p\right)$ is unbounded, that is, $\parallel\phi\left(t,p\right)\parallel_2\rightarrow+\infty$ as $t\rightarrow\pm\infty$.
%	
%Let $\omega\left(p\right)$ be the $\omega$-limit set of $p$.  When  $\phi\left(t,p\right)$ is bounded, the flow $\phi\left(t,p\right)$ tends to the $\omega$-limit point  $q$ in $\omega\left(p\right)$ as $t\rightarrow+\infty$.  Suppose that there exists a point $q_0\in\omega\left(p\right)$  such that $q\neq q_0$.  Let $U_q$ and $U_{q_0}$ be the arbitrarily small neighborhoods  of $q$ and $q_0$, respectively. Note that $\omega\left(p\right)$ is compact and  connected (see Theorem 1.19  of \cite{MR2256001}).
%Then, $U_q\cap\omega\left(p\right)\neq\varnothing$ and $U_{q_0}\cap\omega\left(p\right)\neq\varnothing$.  From Theorem \ref{th12},
%the singular points $q$ and $q_0$ are nonisolated, contradicting the hypothesis. Thus, $\omega\left(p\right)=q$.  This implies that the flow $\phi\left(t,p\right)$ either tends to the infinity of $\mathbb{R}^n$ or  singular point  $q$ as $t\rightarrow+\infty$. Furthermore, I change
%$t\rightarrow+\infty$ to $t\rightarrow-\infty$.
%\end{proof}

\begin{proposition}\label{pr1}
Assume that the polynomial map $F=\left(f_1,\ldots,f_n\right):\mathbb{R}^n\rightarrow \mathbb{R}^n$ satisfies  $F\left(\mathbf{0}\right)=\mathbf{0}$ and $\det DF\left(\mathbf{x}\right)\neq0$ for all $\mathbf{x}\in\mathbb{R}^n$. Let  $H\left(\mathbf{x}\right)=\parallel F\left(\mathbf{x}\right)\parallel_2^2/2$. Consider the following  gradient vector field
\begin{align}\label{eq10}
	&\mathscr{Y}=\left(-\partial_{x_1}H,\cdots,-\partial_{x_n}H\right).
\end{align}
If
\begin{align}\label{eq11}
	&\lim_{\parallel \mathbf{x}\parallel\rightarrow+\infty}\parallel F\left(\mathbf{x}\right)\parallel=+\infty,
\end{align}
then the origin is the unique singular point of  vector field \eqref{eq10}.
\end{proposition}

\begin{proof}
 We claim that for any $p\in \mathbb R^n$, which is not finite singular point, the flow $\varphi(t, p)$ tends to
some finite singular point  when $t\rightarrow +\infty$ and  tends to the infinity when $t\rightarrow -\infty$.
Since $H\left(\mathbf{x}\right)\geq 0$ and
	$$\frac{dH\left(\varphi\left(t,p\right)\right)}{dt}=-\parallel\mathscr{Y}\left(\varphi\left(t,p\right)\right)\parallel_2^2<0,$$
the limits
$$ \lim_{t\rightarrow +\infty}H\left(\varphi\left(t,p\right)\right)\triangleq \alpha, \quad \lim_{t\rightarrow -\infty}H\left(\varphi\left(t,p\right)\right)\triangleq \beta,$$
exist, where $\alpha \in \mathbb  R$ and $\beta$ maybe $+\infty$ $	\left(\beta>H(p)\right)$.

By Proposition \ref{pr2}, the flow $\varphi\left(t,p\right)$  tends to some singular point of $\mathscr{Y}$  or the finite as $t\rightarrow+\infty$.
Since  $\lim_{\parallel \mathbf{x}\parallel\rightarrow+\infty}\parallel F\left(\mathbf{x}\right)\parallel=+\infty$ and $\alpha \in \mathbb  R$,
 the flow $\varphi\left(t,p\right)$  must tend to some singular point of $\mathscr{Y}$  or the finite as $t\rightarrow+\infty$. 

When $t\rightarrow-\infty$,  the flow $\varphi\left(t,p\right)$  must tend to the infinity, else it must tend to some singular point of $\mathscr{Y}$,
denoted by $q$. By Lemma \ref{le2}, $F\left(q\right)=\mathbf{0}$ and $H(q)=0$, which is contradiction to the fact  $\beta=H(q)=0>H(p)\geq 0$. The claim is proved. This means that the infinity of $\mathscr{Y}$  is repeller.

  Let $q_1,\ldots, q_l$ be the all finite  singular points of $\mathscr{Y}$. Using Lemma \ref{le2} and Theorem \ref{th4}, we have
  $$\sum_{j=1}^li_{\mathscr{Y}}\left(q_j\right)=l\left(-1\right)^n=\left(-1\right)^n,$$
  that is, $l=1$.  This proposition holds.
\end{proof}

\begin{proof}[{\bf Proof of Theorem \ref{th3}}]
\emph {Sufficiency.} We will prove the
sufficiency by contradiction.  Suppose that there exist two points $\mathbf{a},\mathbf{b}\in\mathbb{R}^n$ with $\mathbf{a}\neq\mathbf{b}$  such that $F\left(\mathbf{a}\right)=F\left(\mathbf{b}\right)=\mathbf{c}$. We construct a new polynomial map as follows
\begin{align}\label{eq6}
&\overline{F}\left(\mathbf{z}\right)=F\left(\mathbf{z}+\mathbf{b}\right)-\mathbf{c}=\left(f_1\left(\mathbf{z}+\mathbf{b}\right),\ldots,f_n\left(\mathbf{z}+\mathbf{b}\right)\right)-\mathbf{c}.
\end{align}
It is easy to prove that $\overline{F}\left(\mathbf{0}\right)=\overline{F}\left(\mathbf{a-b}\right)=\mathbf{0}$, $\det D\overline{F}\left(\mathbf{z}\right)\neq0$ for all $\mathbf{z}\in\mathbb{R}^n$ and $\parallel F\left(\mathbf{z}\right)\parallel=+\infty$ as $\parallel \mathbf{z}\parallel\rightarrow+\infty$. Let $\overline{H}\left(\mathbf{z}\right)=\parallel \overline{F}\left(\mathbf{z}\right)\parallel_2^2/2$.
Consider the following  gradient vector field
\begin{align}\label{eq12}
&\overline{\mathscr{Y}}=\left(-\partial_{z_1}\overline{H},\cdots,-\partial_{z_n}\overline{H}\right).
\end{align}
By Lemma \ref{le2}, the points $\mathbf{0}$ and $\mathbf{a}-\mathbf{b}\neq\mathbf{0}$ are singular points of $\overline{\mathscr{Y}}$. This contradicts Proposition \ref{pr1}.  So, $F$ is injective.

\emph {Necessity.} \;Let closed ball
$$\mathcal{B}_h\left(\mathbf{0}\right)=\left\{\mathbf{x}\in\mathbb{R}^n\big|\parallel\mathbf{x}\parallel_2^2\leq 2h,\;0\leq h\right\}$$
and hypersphere $\mathbb{S}_h^{n-1}=\partial\mathcal{B}_h\left(\mathbf{0}\right)$.  Then,
$$\mathbb{R}^n=\bigcup\limits_{h\geq0}\mathbb{S}_h^{n-1}.$$

Let $H=\parallel F\left(\mathbf{x}\right)\parallel_2^2/2$.
Since
$$F: \mathbb{R}^n\rightarrow \mathbb{R}^n=\bigcup\limits_{h\geq0}\mathbb{S}_h^{n-1}$$
is injective, we have that  $H^{-1}\left\{h\right\}=F^{-1}\left(\mathbb{S}_h^{n-1}\right)$ is a closed
hypersurface $S_h$. For every $M>0$, one can construct an open region $F^{-1}\left(\mathbb{R}^n\setminus \mathcal{B}_M\left(\mathbf{0}\right)\right)$ such that $H\left(\mathbf{x}\right)>M$ for all $\mathbf{x}\in F^{-1}\left(\mathbb{R}^n\setminus \mathcal{B}_M\left(\mathbf{0}\right)\right)$. Let
$$\delta>\max\limits_{\mathbf{x}\in F^{-1}\left(B_M\left(\mathbf{0}\right)\right)}\parallel \mathbf{x}\parallel_2^2$$
For every $M>0$, there exists a $\delta>0$ such that if $\mathbf{x}\in\mathbb{R}^n\setminus \mathcal{B}_\delta\left(\mathbf{0}\right)$, then $H\left(\mathbf{x}\right)>M$. Thus,
$H\left(\mathbf{x}\right)\rightarrow+\infty$ as $\parallel\mathbf{x}\parallel_2\rightarrow+\infty$.

This completes the proof of Theorem \ref{th3}.
\end{proof}

\section{Algebraic conditions}\label{se2}
This section is devoted to prove Theorem \ref{th10} and Theorem \ref{th2}.
\begin{lemma}\label{le-6}{\rm (see \cite{tian2021necessary})}
	Let $\mathbf{s}=\left(s_1,\ldots,s_n\right)\in\mathbb{N}_+^n$, $\mathbf{x}=\left(x_1,\ldots,x_n\right)\in \mathbb{R}^n$ and $\mathbf{y}=\left(y_1,\ldots,y_n\right)\in\mathbb{R}^n$. If $\parallel\mathbf{y}\parallel_2=1$ and
	\begin{align}\label{eq18}
		&x_1=\frac{y_1}{\rho^{s_1}},\ldots,x_n=\frac{y_n}{\rho^{s_n}},
	\end{align}
	then $\parallel \mathbf{x}\parallel_2\rightarrow+\infty$ if and only if $\rho\rightarrow0$.
\end{lemma}
\noindent The proof of Lemma \ref{le-6} is given in Lemma 7 of \cite{tian2021necessary}.

\begin{lemma}\label{le4}
	Assume that the  polynomial $\mathcal{P}\left(\mathbf{x}\right)$ is  non-negative for all $\mathbf{x}\in\mathbb{R}^n$. If $\partial_{\mathbf{x}}\mathcal{P}\left(\mathbf{x}\right)$ only vanishes at $\mathbf{x}=\mathbf{0}$ in $\mathbb{R}^n$, then $\mathbf{x}=\mathbf{0}$ is also the unique real zero of $\mathcal{P}\left(\mathbf{x}\right)$.
\end{lemma}

\begin{proof}
Suppose that there exists a points $\mathbf{a}\in\mathbb{R}^n$ with $\mathbf{a}\neq\mathbf{0}$  such that $\mathcal{P}\left(\mathbf{a}\right)=0$. Since $\mathcal{P}\left(\mathbf{x}\right)\geq0$ for all  $\mathbf{x}\in\mathbb{R}^n$, $\mathcal{P}\left(\mathbf{x}\right)$ attains its global minimum at a point $\mathbf{a}$. Thus, $\partial_{\mathbf{x}}\mathcal{P}|_{\mathbf{x}=\mathbf{a}}=\mathbf{0}$ with $\mathbf{a}\neq\mathbf{0}$, a contradiction. The lemma is confirmed.
%Consider the following  gradient vector field
%\begin{align}\label{eq13}
%&\left(-\partial_{x_1}\mathcal{P}\left(\mathbf{x}\right),\ldots,-\partial_{x_n}\mathcal{P}\left(\mathbf{x}\right)\right).
%\end{align}
%Suppose that there exists a points $\mathbf{a}\in\mathbb{R}^n$ with $\mathbf{a}\neq\mathbf{0}$  such that $\mathcal{P}\left(\mathbf{a}\right)=0$. Let the flow $\varphi\left(t,\mathbf{a}\right)$ generated by vector field \eqref{eq13}  with $\varphi\left(0,\mathbf{a}\right)=\mathbf{a}$. The polynomial $\mathcal{P}\left(\mathbf{x}\right)$ along the flow $\varphi\left(t,\mathbf{a}\right)$ is decreasing due to
%$$\dfrac{d\mathcal{P}}{dt}|_{\mathbf{x}=\varphi\left(t,\mathbf{a}\right)}=-\sum_{i=1}^n\left(\partial_{x_i}\mathcal{P}\right)^2|_{\mathbf{x}=\varphi\left(t,\mathbf{a}\right)}<0.$$
%This means that there exists a $t_0>0$ such that $0=\mathcal{P}\left(\mathbf{a}\right)>\mathcal{P}\left(\varphi\left(t,\mathbf{a}\right)\right)$  for all $t>t_0$.  This contradicts $\mathcal{P}\left(\mathbf{x}\right)\geq0$ for all  $\mathbf{x}\in\mathbb{R}^n$.
\end{proof}
\begin{remark}\label{re1}
 \emph{The converse of Lemma  \ref{le4} is not true, for example $\mathcal{P}\left(x\right)=x^2\left((9x+10)^2+8\right)\geq0$ for all $x\in \mathbb{R}$, but $\partial_x\mathcal{P}\left(x\right)=108x \left(x+1\right) \left(3x+2\right)$.  This converse is correct for quasi-homogeneous polynomial as follows. }
\end{remark}

\begin{corollary}\label{co2}
	Assume that the quasi-homogeneous polynomial $\mathcal{P}\left(\mathbf{x}\right)$ is  non-negative for all $\mathbf{x}\in\mathbb{R}^n$.  Then $\partial_{\mathbf{x}}\mathcal{P}\left(\mathbf{x}\right)$ only vanishes at $\mathbf{x}=\mathbf{0}$ in $\mathbb{R}^n$ if and only if $\mathbf{x}=\mathbf{0}$ is the unique real zero of $\mathcal{P}\left(\mathbf{x}\right)$.
\end{corollary}

\begin{proof}
Let $\mathcal{P}\left(\mathbf{x}\right)$  be a quasi-homogeneous polynomial of weighted degree $d$ with respect to weight exponents $\mathbf{s}$.  By Lemma \ref{le4}, the necessity is clear.

Applying the generalized Euler formula, we obtain
$$\sum_{i=1}^ns_ix_i\partial_{x_i}\mathcal{P}\left(\mathbf{x}\right)=d\mathcal{P}\left(\mathbf{x}\right).$$
Suppose that $\partial_{\mathbf{x}}\mathcal{P}\left(\mathbf{x}\right)\big|_{\mathbf{x}=\mathbf{x}_0}=\mathbf{0}$ with $\mathbf{x}_0\neq\mathbf{0}$. Then,  $\mathcal{P}\left(\mathbf{x}_0\right)=0$ with $\mathbf{x}_0\neq\mathbf{0}$, a contradiction. The proof is finished.
\end{proof}

%\begin{lemma}\label{le6}
%If the quasi-homogeneous polynomial $\mathcal{P}\left(\mathbf{x}\right)$  has isolated real zeros, then  the real zero of $\mathcal{P}\left(\mathbf{x}\right)$ is $\mathbf{x}=\mathbf{0}$.
%\end{lemma}
%
%\begin{proof}
%	Let $\mathcal{P}\left(\mathbf{x}\right)$  be a quasi-homogeneous polynomial of weighted degree $d$ with respect to weight exponents $\mathbf{s}$,  and $\mathbf{x}=\mathbf{x}_0$ be an isolated real zero of $\mathcal{P}\left(\mathbf{x}\right)$.
%	
%	If $\mathbf{x}_0\neq\mathbf{0}$, then  the set $\left\{\mathbf{x}|\mathbf{x}=\lambda^{\mathbf{s}}\mathbf{x}_0,\;\lambda\in\mathbb{R}\right\}$ is fulfilled of real zeros of $\mathcal{P}\left(\mathbf{x}\right)$,  in contradiction that $\mathbf{x}_0$ is an isolated real zero.  The lemma is proved.
%\end{proof}

\begin{lemma}\label{le7}
Let $\mathcal{P}\left(\mathbf{x}\right)$  be a polynomial with $\mathbf{x}\in\mathbb{R}^n$. If there exists a weight  exponents $\mathbf{s}=\left(s_1,\ldots,s_n\right)\in\mathbb{N}_+^n$  such that
the higher $\mathbf{s}$-quasi-homogeneous term of $\mathcal{P}\left(\mathbf{x}\right)$ has  the unique zero $\mathbf x=\mathbf 0$, then
$$\lim_{\parallel\mathbf{x}\parallel\rightarrow+\infty}\big|\mathcal{P}\left(\mathbf{x}\right)\big|=+\infty.$$
\end{lemma}
\begin{proof}
The polynomial $\mathcal{P}\left(\mathbf{x}\right)$ can be written in the sum of its $\mathbf{s}$-quasi-homogeneous terms, that is,
\begin{align}\label{eq37}
	&\mathcal{P}\left(\mathbf{x}\right)=\sum_{i=1}^{d}\mathcal{P}_i\left(\mathbf{x}\right),
\end{align}
where $\mathcal{P}_i\left(\mathbf{x}\right)$ is quasi-homogeneous of weighted degree $i$ with respect to weight exponents $\mathbf{s}$. Since the quasi-homogeneous polynomial $\mathcal{P}_d\left(\mathbf{x}\right)$ has only the zero $\mathbf x=\mathbf 0$, there exists a $\delta>0$ such that  $\mathcal{P}_d\left(\mathbf{y}\right)\neq0$ for all
$$\mathbf{y}\in\mathbb{S}_\delta^{n-1}=\left\{\mathbf{y}\big|\parallel\mathbf{y}\parallel_2=\delta,\;\mathbf{y}=\left(y_1,\ldots,y_n\right)\right\}.$$
Since $\mathcal{P}_{d}\left(\mathbf{y}\right)$ is continuous on $\mathbf{y}\in\mathbb{S}_\delta^{n-1}$,
there exists a $L>0$ such that $\big|\mathcal{P}_{d}\left(\mathbf{y}\right)\big|\geq L$ for all $\mathbf{y}\in\mathbb{S}_\delta^{n-1}$.

Performing the change of coordinates \eqref{eq18}, we have
\begin{align}\label{eq19}
	&\lim_{\parallel \mathbf{x}\parallel_2\rightarrow+\infty}\dfrac{1}{\big|\mathcal{P}\left(\mathbf{x}\right)\big|}=\lim_{\rho\rightarrow0}\dfrac{\rho^{d}}{\bigg|\sum\limits_{i=1}^{d}\rho ^{d-i} \mathcal{P}_i\left(\mathbf{y}\right)\bigg|}
\end{align}
with $\mathbf{y}\in\mathbb{S}_\delta^{n-1}$. Since
$$\lim_{\rho\rightarrow0}\bigg|\sum_{i=1}^{d}\rho ^{d-i} \mathcal{P}_i\left(\mathbf{y}\right)\bigg|=\big|\mathcal{P}_{d}\left(\mathbf{y}\right)\big|,$$	
there exists a $\tilde{\delta}>0$ such that
\begin{align}\label{eq20}
	&\frac{L}{2}\leq\frac{1}{2}\big|\mathcal{P}_{d}\left(\mathbf{y}\right)\big|<\bigg|\sum_{i=1}^{d}\rho ^{d-i} \mathcal{P}_i\left(\mathbf{y}\right)\bigg|
\end{align}
for all $0<|\rho|<\tilde{\delta}$ and $\mathbf{y}\in\mathbb{S}_{\delta}^{n-1}$. So,
$$0\leq\lim_{\parallel\mathbf{x}\parallel_2\rightarrow+\infty}\dfrac{1}{\big|\mathcal{P}\left(\mathbf{x}\right)\big|}\leq\lim_{\rho\rightarrow0}\dfrac{2\rho^d}{\big|\mathcal{P}_{d}\left(\mathbf{y}\right)\big|}\leq\lim_{\rho\rightarrow0}\dfrac{2\rho^d}{L}=0,\quad\text{that is},\quad \lim_{\parallel\mathbf{x}\parallel\rightarrow+\infty}\big|\mathcal{P}\left(\mathbf{x}\right)\big|=+\infty.$$
This ends the proof.

\end{proof}

\begin{proof}[{\bf Proof of Theorem \ref{th10}}]
Let
\begin{align}\label{eq17}
&H\left(\mathbf{x}\right)=\dfrac{\parallel F\left(\mathbf{x}\right)\parallel_2^2}{2}=\sum_{i=2}^{2d}H_i\left(\mathbf{x}\right)\geq0,
\end{align}
where $H_i\left(\mathbf{x}\right)$ is quasi-homogeneous polynomial of weighted degree $i$ with respect to weight exponents $\mathbf{s}$.
Clearly, $H_\mathbf{s}\left(\mathbf{x}\right)=H_{2d}\left(\mathbf{x}\right)$ and $H_{2d}\left(\mathbf{x}\right)\geq0$ for all $\mathbf{x}\in \mathbb{R}^n$.  Since $\partial_{\mathbf{x}}H_\mathbf{s}\left(\mathbf{x}\right)=\partial_{\mathbf{x}}H_{2d}\left(\mathbf{x}\right)$ only vanishes at the origin in $\mathbb{R}^n$, by Lemma \ref{le4}, $\mathbf{x}=\mathbf{0}$ is the unique real zero  of $H_{2d}\left(\mathbf{x}\right)$. Using Lemma \ref{le7}, we have
$$\lim_{\parallel\mathbf{x}\parallel_2\rightarrow+\infty}H\left(\mathbf{x}\right)=+\infty.$$
From Theorem \ref{th3}, it follows that $F$ is injective.
\end{proof}

To prove Theorem \ref{th2}, we introduce  the following notations.

Using the multi-index notations, the polynomial $\mathcal{P}\left(\mathbf{x}\right)\in\mathbb{R}\left[\mathbf{x}\right]$ can be written as
\begin{align}\label{eq50}
&\mathcal{P}\left(\mathbf{x}\right)=\sum_{\mathbf{k}\in\mathbb{N}^n}c_{\mathbf{k}}\mathbf{x}^{\mathbf{k}},
\end{align}
where $\mathbf{k}=\left(k_1,\ldots,k_n\right)\in\mathbb{N}^n$, $\mathbf{x}^{\mathbf{k}}=x_1^{k_1}\cdots x_n^{k_n}$ and $c_\mathbf{k}\in\mathbb{R}$. The degree of monomial $\mathbf{x}^{\mathbf{k}}=x_1^{k_1}\cdots x_n^{k_n}$ is
$$\big|\mathbf{k}\big|= \sum_{i=1}^nk_i.$$

Let $\lessdot,\gtrdot$ be the inner product of two vectors. If the polynomial $\mathcal{P}\left(\mathbf{x}\right)\in\mathbb{R}\left[\mathbf{x}\right]$ is quasi-homogeneous of weighted degree $\ell$ with respect to weight exponents $\mathbf{s}=\left(s_1,\ldots,s_n\right)\in\mathbb{N}_+^n$, then
\begin{align}\label{eq51}
	&\mathcal{P}\left(\mathbf{x}\right)=\sum_{\mathbf{k}\in\mathbb{N}^n\;\text{and}\;<\mathbf{s},\mathbf{k}>=\ell}c_{\mathbf{k}}\mathbf{x}^{\mathbf{k}}.
\end{align}
As we know, the Euclidean space  $\mathbb{R}^n$ can be represented in Cartesian product
$\mathbb{R}^n=\mathbb{R}^{n_1}\times\cdots\times\mathbb{R}^{n_r}$ with $\sum_{i=1}^rn_i=n$. Similarly, $\mathbb{N}^n=\mathbb{N}^{n_1}\times\cdots\times\mathbb{N}^{n_r}$ and $\mathbb{N}_+^n=\mathbb{N}_+^{n_1}\times\cdots\times\mathbb{N}_+^{n_r}$  with $\sum_{i=1}^rn_i=n$. The vectors $\mathbf{x}\in \mathbb{R}^n$, $\mathbf{k}\in\mathbb{N}^n$ and $\mathbf{s}\in\mathbb{N}_+^n$ can be written as
\begin{align*}
&\mathbf{x}=\left(\mathbf{x}_1,\ldots,\mathbf{x}_r\right)\in\mathbb{R}^{n_1}\times\cdots\times\mathbb{R}^{n_r},\;\mathbf{k}=\left(\mathbf{k}_1,\ldots,\mathbf{k}_r\right)\in\mathbb{N}^{n_1}\times\cdots\times\mathbb{N}^{n_r}
\end{align*}
and     $$\mathbf{s}=\left(\mathbf{s}_1,\ldots,\mathbf{s}_r\right)\in\mathbb{N}_+^{n_1}\times\cdots\times\mathbb{N}_+^{n_r},$$
respectively, where $\mathbf{x}_i\in\mathbb{R}^{n_i}$, $\mathbf{k}_i\in\mathbb{N}^{n_i}$, $\mathbf{s}_i\in\mathbb{N}_+^{n_i}$ and $\sum_{i=1}^rn_i=n$. Equation \eqref{eq50} also can be rewritten as
\begin{align*}
&\mathcal{P}\left(\mathbf{x}\right)=\mathcal{P}\left(\mathbf{x}_1,\ldots,\mathbf{x}_r\right)=\sum_{\left(\mathbf{k}_1,\ldots,\mathbf{k}_r\right)\in\mathbb{N}^{n_1}\times\cdots\times\mathbb{N}^{n_r}}c_{\left(\mathbf{k}_1,\ldots,\mathbf{k}_r\right)}\mathbf{x}_1^{\mathbf{k}_1}\cdots\mathbf{x}_r^{\mathbf{k}_r}
\end{align*}
with $c_{\left(\mathbf{k}_1,\ldots,\mathbf{k}_r\right)}\in\mathbb{R}$. Recall that $\partial_{\mathbf{x}}\mathcal{P}\left(\mathbf{x}\right)=\left(\partial_{x_1}\mathcal{P}\left(\mathbf{x}\right),\ldots,\partial_{x_n}\mathcal{P}\left(\mathbf{x}\right)\right)$.  We can also write $\partial_{\mathbf{x}}\mathcal{P}\left(\mathbf{x}\right)=\left(\partial_{\mathbf{x}_1}\mathcal{P}\left(\mathbf{x}\right),\ldots,\partial_{\mathbf{x}_r}\mathcal{P}\left(\mathbf{x}\right)\right)$ with $\mathbf{x}=\left(\mathbf{x}_1,\dots,\mathbf{x}_r\right)\in\mathbb{R}^n$.

Correspondingly, the  quasi-homogeneous  polynomial \eqref{eq51} can be expressed as
\begin{align*}
&\mathcal{P}\left(\mathbf{x}\right)=\sum\limits_{<\mathbf{s}_1,\mathbf{k}_1>+\cdots+<\mathbf{s}_r,\mathbf{k}_r>=\ell}c_{\left(\mathbf{k}_1,\ldots,\mathbf{k}_r\right)}\mathbf{x}_1^{\mathbf{k}_1}\cdots\mathbf{x}_r^{\mathbf{k}_r},
\end{align*}
where $\left(\mathbf{k}_1,\ldots,\mathbf{k}_r\right)\in\mathbb{N}^n=\mathbb{N}^{n_1}\times\cdots\times\mathbb{N}^{n_r}$, $\left(\mathbf{s}_1,\ldots,\mathbf{s}_r\right)\in\mathbb{N}_+^n=\mathbb{N}_+^{n_1}\times\cdots\times\mathbb{N}_+^{n_r}$ and $c_{\left(\mathbf{k}_1,\ldots,\mathbf{k}_r\right)}\in\mathbb{R}$.

\begin{proof}[{\bf Proof of Theorem \ref{th2}}]
Let
\begin{align}\label{eq26}
&H\left(\mathbf{x}\right)=\dfrac{\parallel F\left(\mathbf{x}\right)\parallel_2^2}{2}=\dfrac{f_1^2+\cdots+f_n^2}{2}=\sum_{j=2}^{2d}H_j\left(\mathbf{x}\right),
\end{align}
where $H_j\left(\mathbf{x}\right)$ is quasi-homogeneous of weighted degree $j$ with respect to weight exponents $\mathbf{s}$.

Let $\mathscr{Y}_\mathbf{s}$ be the higher $\mathbf{s}$-quasi-homogeneous part of vector field $\mathscr{Y}$. Then, there exist the $\mathbf{s}$-quasi-homogeneous terms of $H\left(\mathbf{x}\right)$, that is,  $H_{i_1}\left(\mathbf{x}\right),\ldots,H_{i_n}\left(\mathbf{x}\right)$ such that
\begin{align}\label{eq25}
&\mathscr{Y}_\mathbf{s}=-\left(\partial_{x_1}H_{i_1},\partial_{x_2}H_{i_2},\cdots,\partial_{x_n}H_{i_n}\right).
\end{align}
Without loss of generality, we can suppose that $2d=i_1 \geq i_2 \geq \ldots \geq i_n$. Obviously one can find positive integers $r$ and $n_1, n_2, \ldots, n_r$ such that $\sum_{i=1}^rn_i=n$ and
$$i_1=i_2=\cdots=i_{n_1}>i_{n_1+1}=\cdots=i_{n_1+n_2}>\cdots>i_{\sum_{i=1}^{r-1}n_{i}+1}=\cdots=i_{\sum_{i=1}^{r}n_{i}}.$$

Denote by $\mathbf{x}_1=(x_{1}, x_{2}, \ldots, x_{n_1}), m_1=i_1$, $\mathbf{x}_2=(x_{n_1+1}, x_{n_1+2}, \ldots, x_{n_1+n_2}), m_2=i_{n_1+1}, \ldots,$ 
then   $\mathbf{x}=\left(\mathbf{x}_1,\ldots,\mathbf{x}_r\right)\in\mathbb{R}^{n_1}\times\cdots\times\mathbb{R}^{n_r}$, and equation \eqref{eq25}
can be written as
$$\mathscr{Y}_\mathbf{s}=-\left(\partial_{\mathbf{x}_1}H_{m_1},\partial_{\mathbf{x}_2}H_{m_2},\cdots,\partial_{\mathbf{x}_r}H_{m_r}\right).$$

 For convenience,  we denote
$$\overline{\mathscr{Y}_\mathbf{s}}:=-\mathscr{Y}_\mathbf{s}=\left(\partial_{\mathbf{x}_1}H_{m_1},\partial_{\mathbf{x}_2}H_{m_2},\cdots,\partial_{\mathbf{x}_r}H_{m_r}\right).$$
Thereby, $2d=m_1>\cdots>m_r$,
\begin{align}\label{eq35}
&\mathbf{x}=\left(\overset{\mathbf{x}_1}{\overbrace{x_1,\ldots,x_{n_1}}},\overset{\mathbf{x}_2}{\overbrace{x_{n_1+1},\ldots}},\overset{\ldots}{\overbrace{\cdots}},\overset{\mathbf{x}_r}{\overbrace{\ldots,x_{n-1},x_n}}\right)\in\mathbb{R}^{n_1}\times\cdots\times\mathbb{R}^{n_r},
\end{align}
and
\begin{align}\label{eq36}
&\mathbf{s}=\left(\overset{\mathbf{s}_1}{\overbrace{s_1,\ldots,s_{n_1}}},\overset{\mathbf{s}_2}{\overbrace{s_{n_1+1},\ldots}},
\overset{\ldots}{\overbrace{\cdots}},\overset{\mathbf{s}_r}{\overbrace{\ldots,s_{n-1},s_n}}\right)\in\mathbb{N}_+^{n_1}\times\cdots\times\mathbb{N}_+^{n_r}.
\end{align}

Next, we distinguish between the cases $r=1$ and $r>1$.

{\bf Case 1: $r=1$.} For this case, equation \eqref{eq25} becomes $\mathscr{Y}_\mathbf{s}=-\partial_{\mathbf{x}}H_{2d}$. Since $\overline{\mathscr{Y}_\mathbf{s}}=\partial_{\mathbf{x}}H_{2d}$ only vanishes at the origin, by Theorem \ref{th10}, $F$ is injective.

{\bf Case 2: $r>1$.}   In this case, we have that
\begin{equation}\label{eq27}
H_j	\left(\mathbf{x}\right)=
\begin{cases}
	H_{j}\left(\mathbf{x}_1,\ldots,\mathbf{x}_i\right),\;\text{if}\;m_{i+1}< j \leq m_i\;\text{with}\;i=1,\ldots,r-1;\\
	H_j\left(\mathbf{x}_1,\ldots,\mathbf{x}_r\right),\;\text{if}\; 2\leq j\leq m_r;
\end{cases}
\end{equation}
where $H_j\left(\mathbf{x}\right)$ is quasi-homogeneous of weighted degree $j$ with respect to weight exponents $\mathbf{s}$. Then, equation \eqref{eq26} can be written as
\begin{equation}\label{eq28}
	\begin{split}
		H\left(\mathbf{x}_1,\ldots,\mathbf{x}_r\right)=&\dfrac{f_1^2+\cdots+f_n^2}{2}=\sum_{m_2< j\leq 2d}H_j\left(\mathbf{x}_1\right)+\sum_{m_3< j\leq m_2}H_j\left(\mathbf{x}_1,\mathbf{x}_2\right)+\cdots\\
		&+\sum_{m_r<j\leq m_{r-1}}H_j\left(\mathbf{x}_1,\ldots,\mathbf{x}_{r-1}\right)+\sum_{2\leq j\leq m_r}H_j\left(\mathbf{x}_1,\ldots,\mathbf{x}_r\right).
	\end{split}
\end{equation}
Note that each quasi-homogeneous  polynomial $H_{m_i}\left(\mathbf{x}\right)$ can be regard as a  polynomial  in the variable $\mathbf{x}_i$, that is,
\begin{align}\label{eq46}
H_{m_i}\left(\mathbf{x}_1,\ldots,\mathbf{x}_i\right)=\mathscr{H}_{m_i}\left(\mathbf{x}_i\right)+\mathcal{H}_{m_i}\left(\mathbf{x}_1,\ldots,\mathbf{x}_i\right),
\end{align}
where $\mathscr{H}_{m_i}\left(\mathbf{x}_i\right)$ is the higher-degree term of $H_{m_i}\left(\mathbf{x}\right)$ in the variable $\mathbf{x}_i$ and $i=1,\ldots,r$.

The following two claims will finish the proof of {\bf Case 2}.

{\bf Claim 1.} \emph{For $i=1,\ldots,r$, $\mathscr{H}_{m_i}\left(\mathbf{x}_i\right)\geq0$ and $\mathscr{H}_{m_i}\left(\mathbf{x}_i\right)=0$ if and only if $\mathbf{x}_i=\mathbf{0}_i\in\mathbb{R}^{n_i}$.}

The proof goes by induction on $i$. Taking  $\mathbf{x}=\left(\mathbf{0}_1,\ldots,\mathbf{0}_{r-1},\mathbf{x}_r\right)$, one can obtain that
\begin{align}\label{eq29}
&\overline{\mathscr{Y}_\mathbf{s}}\big|_{\mathbf{x}=\left(\mathbf{0}_1,\ldots,\mathbf{0}_{r-1},\mathbf{x}_r\right)}=\left(\mathbf{0}_1,\cdots,\mathbf{0}_{r-1}, \partial_{\mathbf{x}_r}\mathscr{H}_{m_r}\right)
\end{align}
and
\begin{small}
\begin{equation}\label{eq30}
	H\left(\mathbf{0}_1,\ldots,\mathbf{0}_{r-1},\mathbf{x}_r\right)=\dfrac{f_1^2+\cdots+f_n^2}{2}\Bigg|_{\mathbf{x}=\left(\mathbf{0}_1,\ldots,\mathbf{0}_{r-1},\mathbf{x}_r\right)}=\mathscr{H}_{m_r}\left(\mathbf{x}_r\right)+\sum_{2\leq j< m_r}H_j\left(\mathbf{0}_1,\ldots,\mathbf{0}_{r-1},\mathbf{x}_r\right).
\end{equation}
\end{small}Obviously, $\mathscr{H}_{m_r}\left(\mathbf{x}_r\right)\geq0$.  Since equation \eqref{eq29} only vanishes at the origin in $\mathbb{R}^n$, we have that $\partial_{\mathbf{x}_r}\mathscr{H}_{m_r}$ only vanishes at $\mathbf{x}_r=\mathbf{0}_r$.
By Lemma \ref{le4}, $\mathscr{H}_{m_r}\left(\mathbf{x}_r\right)=0$ if and only if $\mathbf{x}_r=\mathbf{0}_r\in\mathbb{R}^{n_r}$.
The claim  is confirmed for $i=r$.

We assume that the claim holds for $i=i_0,\ldots, r$.  Substituting
$$ \mathbf{x}=\left(\mathbf{0}_1,\ldots,\mathbf{0}_{i_0-2},\mathbf{x}_{i_0-1},\mathbf{x}_{i_0},\ldots,\mathbf{x}_r\right)$$
into equations \eqref{eq25} and \eqref{eq28}, we get, respectively,
\begin{small}
\begin{align}\label{eq31}
	&\overline{\mathscr{Y}_\mathbf{s}}\big|_{\mathbf{x}=\left(\mathbf{0}_1,\ldots,\mathbf{0}_{i_0-2},\mathbf{x}_{i_0-1},\mathbf{x}_{i_0},\ldots,\mathbf{x}_r\right)}=\left(\mathbf{0}_1,\cdots,\mathbf{0}_{i_0-2}, \partial_{\mathbf{x}_{i_0-1}}\mathscr{H}_{m_{i_0-1}},\partial_{\mathbf{x}_{i_0}}H_{m_{i_0}},\ldots,\partial_{\mathbf{x}_r}H_{m_r}\right)
\end{align}
\end{small}and
\begin{equation}\label{eq32}
	\begin{split}
 H\left(\mathbf{0}_1,\ldots,\mathbf{0}_{i_0-2},\mathbf{x}_{i_0-1},\mathbf{x}_{i_0},\ldots,\mathbf{x}_r\right)&=\dfrac{f_1^2+\cdots+f_n^2}{2}\Bigg|_{\mathbf{x}=\left(\mathbf{0}_1,\ldots,\mathbf{0}_{i_0-2},\mathbf{x}_{i_0-1},\mathbf{x}_{i_0},\ldots,\mathbf{x}_r\right)}\\
 &=\mathscr{H}_{m_{i_0-1}}\left(\mathbf{x}_{i_0-1}\right)+G\left(\mathbf{x}_{i_0-1},\mathbf{x}_{i_0},\ldots,\mathbf{x}_r\right),
	\end{split}
\end{equation}
where
\begin{small}
\begin{equation}\label{eq33}
	\begin{split}
	&G\left(\mathbf{x}_{i_0-1},\mathbf{x}_{i_0},\ldots,\mathbf{x}_r\right)=\sum_{m_{i_0}<j<m_{i_0-1}}H_j\left(\mathbf{0}_1,\ldots,\mathbf{0}_{i_0-2},\mathbf{x}_{i_0-1}\right)+\sum_{m_{i_0+1}<j\leq m_{i_0}}H_j\left(\mathbf{0}_1,\ldots,\mathbf{0}_{i_0-2},\mathbf{x}_{i_0-1},\mathbf{x}_{i_0}\right)+\\
&\cdots+\sum_{m_r<j\leq m_{r-1}}H_j\left(\mathbf{0}_1,\ldots,\mathbf{0}_{i_0-2},\mathbf{x}_{i_0-1},\mathbf{x}_{i_0},\ldots,\mathbf{x}_{r-1}\right)+\sum_{2\leq j\leq m_r}H_j\left(\mathbf{0}_1,\ldots,\mathbf{0}_{i_0-2},\mathbf{x}_{i_0-1},\mathbf{x}_{i_0},\ldots,\mathbf{x}_r\right).
	\end{split}
\end{equation}
\end{small}Clearly, $\mathscr{H}_{m_{i_0-1}}\left(\mathbf{x}_{i_0-1}\right)\geq0$ for $\mathbf{x}_{i_0-1}\in\mathbb{R}^{n_{i_0-1}}$.

By contradiction,  we suppose that there exists a point $\mathbf{0}_{i_0-1}\neq\mathbf{a}_{i_0-1}\in\mathbb{R}^{n_{i_0-1}}$ such that $\mathscr{H}_{m_{i_0-1}}\left(\mathbf{a}_{i_0-1}\right)=0$. From Corollary \ref{co2}, it follows  that there exists a point $\mathbf{0}_{i_0-1}\neq\mathbf{b}_{i_0-1}\in\mathbb{R}^{n_{i_0-1}}$ such that
\begin{align}\label{eq34}
&\partial_{\mathbf{x}_{i_0-1}}H_{m_{i_0-1}}\big|_{\mathbf{x}=\left(\mathbf{0}_1,\ldots,\mathbf{0}_{i_0-2},\mathbf{b}_{i_0-1},\mathbf{x}_{i_0},\ldots,\mathbf{x}_r\right)}=\partial_{\mathbf{x}_{i_0-1}}\mathscr{H}_{m_{i_0-1}}\big|_{\mathbf{x}_{i_0-1}=\mathbf{b}_{i_0-1}}=\mathbf{0}_{i_0-1}.
\end{align}

Let
$$G_{m_{i_0}}\left(\mathbf{x}_{i_0}\right)=H_{m_{i_0}}\left(\mathbf{0}_1,\ldots,\mathbf{0}_{i_0-2},\mathbf{b}_{i_0-1},\mathbf{x}_{i_0}\right).$$
Then, $G_{m_{i_0}}\left(\mathbf{x}_{i_0}\right)$ is a polynomial in the variable $\mathbf{x}_{i_0}$. It is obvious that  the higher $\mathbf{s}_{i_0}$-quasi-homogeneous term of $G_{m_{i_0}}\left(\mathbf{x}_{i_0}\right)$  is $\mathscr{H}_{m_{i_0}}\left(\mathbf{x}_{i_0}\right)$, where weight exponents $\mathbf{s}_{i_0}$ is given by equation \eqref{eq36}.  By  the induction assumption and Lemma \ref{le7}, we have
$$\lim_{\parallel\mathbf{x}_{i_0}\parallel\rightarrow+\infty}G_{m_{i_0}}\left(\mathbf{x}_{i_0}\right)=+\infty.$$
This means that $G_{m_{i_0}}\left(\mathbf{x}_{i_0}\right)$ has global minimum  in $\mathbb{R}^{n_{i_0}}$.  Therefore, there exists $\mathbf{b}_{i_0}\in\mathbb{R}^{n_{i_0}}$ such that
\begin{align}\label{eq38}
&\partial_{\mathbf{x}_{i_0}}G_{m_{i_0}}\left(\mathbf{x}_{i_0}\right)\big|_{\mathbf{x}_{i_0}=\mathbf{b}_{i_0}}=\partial_{\mathbf{x}_{i_0}}H_{m_{i_0}}\left(\mathbf{0}_1,\ldots,\mathbf{0}_{i_0-2},\mathbf{b}_{i_0-1},\mathbf{x}_{i_0}\right)\big|_{\mathbf{x}_{i_0}=\mathbf{b}_{i_0}}=\mathbf{0}_{i_0}\in\mathbb{R}^{n_{i_0}}.
\end{align}

Let
$$G_{m_{i_0+1}}\left(\mathbf{x}_{i_0+1}\right)=H_{m_{i_0+1}}\left(\mathbf{0}_1,\ldots,\mathbf{0}_{i_0-2},\mathbf{b}_{i_0-1},\mathbf{b}_{i_0},\mathbf{x}_{i_0+1}\right).$$
The higher $\mathbf{s}_{i_0+1}$-quasi-homogeneous term of $G_{m_{i_0+1}}\left(\mathbf{x}_{i_0+1}\right)$  is $\mathscr{H}_{m_{i_0+1}}\left(\mathbf{x}_{i_0+1}\right)$, where weight exponents $\mathbf{s}_{i_0+1}$ is given by equation \eqref{eq36}. Analogously,
$$\lim_{\parallel\mathbf{x}_{i_0+1}\parallel\rightarrow+\infty}G_{m_{i_0+1}}\left(\mathbf{x}_{i_0+1}\right)=+\infty.$$
The polynomial  $G_{m_{i_0+1}}\left(\mathbf{x}_{i_0+1}\right)$ has global minimum  in $\mathbb{R}^{n_{i_0+1}}$. So, there exists $\mathbf{b}_{i_0+1}\in\mathbb{R}^{n_{i_0+1}}$ such that
\begin{small}
\begin{equation}\label{eq39}
	\begin{split}
		\partial_{\mathbf{x}_{i_0+1}}G_{m_{i_0+1}}\left(\mathbf{x}_{i_0+1}\right)\big|_{\mathbf{x}_{i_0+1}=\mathbf{b}_{i_0+1}}&=\partial_{\mathbf{x}_{i_0+1}}H_{m_{i_0+1}}\left(\mathbf{0}_1,\ldots,\mathbf{0}_{i_0-2},\mathbf{b}_{i_0-1},\mathbf{b}_{i_0},\mathbf{x}_{i_0+1}\right)\big|_{\mathbf{x}_{i_0+1}=\mathbf{b}_{i_0+1}}\\
		&=\mathbf{0}_{i_0+1}\in\mathbb{R}^{n_{i_0+1}}.
	\end{split}
\end{equation}
\end{small}
Repeating the above process, for $i=i_0,\ldots,r$, we can obtain that
\begin{equation}\label{eq40}
\begin{array}{c}
\partial_{\mathbf{x}_{i_0}}H_{m_{i_0}}\left(\mathbf{0}_1,\ldots,\mathbf{0}_{i_0-2},\mathbf{b}_{i_0-1},\mathbf{x}_{i_0}\right)\big|_{\mathbf{x}_{i_0}=\mathbf{b}_{i_0}}=\mathbf{0}_{i_0}\in\mathbb{R}^{n_{i_0}},\vspace{2ex}\\
\partial_{\mathbf{x}_{i_0+1}}H_{m_{i_0+1}}\left(\mathbf{0}_1,\ldots,\mathbf{0}_{i_0-2},\mathbf{b}_{i_0-1},\mathbf{b}_{i_0},\mathbf{x}_{i_0+1}\right)\big|_{\mathbf{x}_{i_0+1}=\mathbf{b}_{i_0+1}}=\mathbf{0}_{i_0+1}\in\mathbb{R}^{n_{i_0+1}},\\
	\vdots\\
\partial_{\mathbf{x}_r}H_{m_r}\left(\mathbf{0}_1,\ldots,\mathbf{0}_{i_0-2},\mathbf{b}_{i_0-1},\mathbf{b}_{i_0},\ldots,\mathbf{b}_{r-1},\mathbf{x}_r\right)\big|_{\mathbf{x}_r=\mathbf{b}_r}=\mathbf{0}_r\in\mathbb{R}^{n_r}.\\
\end{array}
\end{equation}
Note that
$$\mathbf{b}:=\left(\mathbf{0}_1,\ldots,\mathbf{0}_{i_0-2},\mathbf{b}_{i_0-1},\mathbf{b}_{i_0},\ldots,\mathbf{b}_{r-1},\mathbf{b}_r\right)\neq\mathbf{0}$$
because $\mathbf{b}_{i_0-1}\neq \mathbf{0}_{i_0-1}$.  By equations  \eqref{eq34} and \eqref{eq40},  the higher $\mathbf{s}$-quasi-homogeneous part of vector field $\mathscr{Y}$ vanishes at point $\mathbf{b}$, that is,
\begin{align}\label{eq41}
	&\overline{\mathscr{Y}_\mathbf{s}}\big|_{\mathbf{x}=\mathbf{b}=\left(\mathbf{0}_1,\ldots,\mathbf{0}_{i_0-2},\mathbf{b}_{i_0-1},\ldots,\mathbf{b}_r\right)\neq\mathbf{0}}=\left(\mathbf{0}_1,\cdots,\mathbf{0}_{i_0-2},\mathbf{0}_{i_0-1},\ldots,\mathbf{0}_r\right).
\end{align}
This contradicts the condition of the theorem. So, $\mathscr{H}_{m_{i_0-1}}\left(\mathbf{x}_{i_0-1}\right)=0$ if and only if $\mathbf{x}_{i_0-1}=\mathbf{0}_{i_0-1}\in\mathbb{R}^{n_{i_0-1}}$.  By induction, the proof of {\bf Claim 1} is finished.

Let
$$m=\prod_{i=1}^rm_i.$$
Construct a new weight exponents $\mathbf{\tilde{s}}$  as follows:
\begin{align}\label{eq42}
&\mathbf{\tilde{s}}=\left(\frac{m}{m_1}\mathbf{s}_1,\frac{m}{m_2}\mathbf{s}_2,\ldots,\frac{m}{m_r}\mathbf{s}_r\right).
\end{align}
{\bf Claim 2.} \emph{The higher $\mathbf{\tilde{s}}$-quasi-homogeneous term of $H\left(\mathbf{x}\right)$ is
\begin{align}\label{eq43}
&\mathscr{H}\left(\mathbf{x}\right)=\sum_{i=1}^r\mathscr{H}_{m_i}\left(\mathbf{x}_i\right),
\end{align}}
where $\mathscr{H}_{m_i}\left(\mathbf{x}_i\right)$ is given in \eqref{eq46} for $i=1,\dots,r$.

From the definition of the quasi-homogeneous polynomial, equation \eqref{eq27} can be represented in the form
\begin{equation}\label{eq45}
	\begin{split}
&H_j	\left(\mathbf{x}\right)=\\
&\begin{cases}
	\sum\limits_{<\mathbf{s}_1,\mathbf{k}_1>+\cdots+<\mathbf{s}_i,\mathbf{k}_i>=j}c_{\left(\mathbf{k}_1,\ldots,\mathbf{k}_i\right)}\mathbf{x}_1^{\mathbf{k}_1}\cdots\mathbf{x}_i^{\mathbf{k}_i},\;\text{if}\;m_{i+1}< j \leq m_i\;\text{with}\;i=1,\ldots,r-1;\vspace{2ex}\\
	\sum\limits_{<\mathbf{s}_1,\mathbf{k}_1>+\cdots+<\mathbf{s}_r,\mathbf{k}_r>=j}c_{\left(\mathbf{k}_1,\ldots,\mathbf{k}_r\right)}\mathbf{x}_1^{\mathbf{k}_1}\cdots\mathbf{x}_r^{\mathbf{k}_r},\;\text{if}\; 2\leq j\leq m_r.
\end{cases}
	\end{split}
\end{equation}
Thereby,
\begin{small}
\begin{equation}\label{eq44}
	\begin{split}
		&H_j\left(\lambda^{\mathbf{\tilde{s}}}\mathbf{x}\right)=\\
		&\begin{cases}
			\sum\limits_{<\mathbf{s}_1,\mathbf{k}_1>+\cdots+<\mathbf{s}_i,\mathbf{k}_i>=j}c_{\left(\mathbf{k}_1,\ldots,\mathbf{k}_i\right)}\lambda^{m\left(<\mathbf{s}_1/m_1,\mathbf{k}_1>+\cdots+<\mathbf{s}_i/m_i,\mathbf{k}_i>\right)}\mathbf{x}_1^{\mathbf{k}_1}\cdots\mathbf{x}_i^{\mathbf{k}_i},\;\text{if}\;m_{i+1}< j \leq m_i\vspace{2ex}\\
			\text{with}\;i=1,\ldots,r-1;\vspace{2ex}\\
			\sum\limits_{<\mathbf{s}_1,\mathbf{k}_1>+\cdots+<\mathbf{s}_r,\mathbf{k}_r>=j}c_{\left(\mathbf{k}_1,\ldots,\mathbf{k}_r\right)}\lambda^{m\left(<\mathbf{s}_1/m_1,\mathbf{k}_1>+\cdots+<\mathbf{s}_r/m_r,\mathbf{k}_r>\right)}\mathbf{x}_1^{\mathbf{k}_1}\cdots\mathbf{x}_r^{\mathbf{k}_r},\;\text{if}\; 2\leq j\leq m_r.
		\end{cases}
	\end{split}
\end{equation}
\end{small}Since $m_1>m_2>\cdots>m_r$, we have $$m\left(<\dfrac{\mathbf{s}_1}{m_1},\mathbf{k}_1>+\cdots+<\dfrac{\mathbf{s}_i}{m_i},\mathbf{k}_i>\right)\leq \dfrac{m}{m_i}\left(<\mathbf{s}_1,\mathbf{k}_1>+\cdots+<\mathbf{s}_i,\mathbf{k}_i>\right)=\dfrac{jm}{m_i}.$$
From equations \eqref{eq46} and \eqref{eq44}, it follows  that  $\deg_{\mathbf{\tilde{s}}}H_{m_1}=\deg_{\mathbf{\tilde{s}}}H_{m_2}=\cdots=\deg_{\mathbf{\tilde{s}}}H_{m_r}=m$
and $\deg_{\mathbf{\tilde{s}}}H_j<m$ with $j\neq m_1,\ldots,m_r$, that is, $\deg_{\mathbf{\tilde{s}}}H=m$.

Using equations \eqref{eq46} and \eqref{eq44}, we get that
\begin{small}
\begin{equation}\label{eq47}
	\begin{split}
		&H_{m_i}\left(\lambda^{\frac{m}{m_1}\mathbf{s}_1}\mathbf{x}_1,\ldots,\lambda^{\frac{m}{m_i}\mathbf{s}_i}\mathbf{x}_i\right)\\
		&=\lambda^m\mathscr{H}_{m_i}\left(\mathbf{x}_i\right)+\sum\limits_{
			\begin{tiny}
				\begin{array}{c}
					<\mathbf{s}_1,\mathbf{k}_1>+\cdots+<\mathbf{s}_i,\mathbf{k}_i>=m_i\\
					\text{with}\;<\mathbf{s}_i,\mathbf{k}_i>\neq m_i
				\end{array}
			\end{tiny}
		}c_{\left(\mathbf{k}_1,\ldots,\mathbf{k}_i\right)}\lambda^{m\left(<\mathbf{s}_1/m_1,\mathbf{k}_1>+\cdots+<\mathbf{s}_i/m_i,\mathbf{k}_i>\right)}\mathbf{x}_1^{\mathbf{k}_1}\cdots\mathbf{x}_i^{\mathbf{k}_i}
	\end{split}
\end{equation}
\end{small}for $i=1,\ldots,r$.  Since $<\mathbf{s}_1,\mathbf{k}_1>+\cdots+<\mathbf{s}_i,\mathbf{k}_i>=m_i$, $<\mathbf{s}_i,\mathbf{k}_i>\neq m_i$ and $m_1>m_2>\cdots>m_r$,  one can obtain that
\begin{small}
$$<\dfrac{\mathbf{s}_1}{m_1},\mathbf{k}_1>+\cdots+<\dfrac{\mathbf{s}_i}{m_i},\mathbf{k}_i>\leq \dfrac{m_i-<\mathbf{s}_i,\mathbf{k}_i>}{m_{i-1}}+\dfrac{<\mathbf{s}_i,\mathbf{k}_i>}{m_i}=\dfrac{m_i}{m_{i-1}}+<\mathbf{s}_i,\mathbf{k}_i>\left(\frac{1}{m_i}-\frac{1}{m_{i-1}}\right)<1$$
\end{small}for $i=1,\ldots,r$. This means that the higher $\mathbf{\tilde{s}}$-quasi-homogeneous term of  $H_{m_i}\left(\mathbf{x}_1,\ldots,\mathbf{x}_i\right)$ is $\mathscr{H}_{m_i}\left(\mathbf{x}_i\right)$ for $i=1,\ldots,r$.
Consequently, {\bf Claim 2}  holds.

By {\bf Claim 1}, we know that $\mathbf{x}=\mathbf{0}$ is a unique real zero of the higher $\mathbf{\tilde{s}}$-quasi-homogeneous term of $H\left(\mathbf{x}\right)$. Applying Theorem \ref{th10}, $F$ is injective.

We complete the proof of Theorem \ref{th2}.
\end{proof}
\section{Proof of Theorems \ref{th13} and \ref{th11}}\label{se6}

In this section, we prove Theorem \ref{th13} and Theorem \ref{th11}.
\begin{proof}[{\bf Proof of Theorem \ref{th13}}]
	Following the proof of Theorem \ref{th2}, we take the weight exponents $\mathbf{\tilde{s}}$  as follows:
	\begin{small}
		\begin{align}\label{eq60}
			&\mathbf{\tilde{s}}=
			\begin{cases}
				\mathbf{s}=\left(s_1,\ldots,s_n\right),\;\text{for {\bf Case 1}  in proof of Theorem \ref{th2}};\\
				\left(\frac{m}{m_1}\mathbf{s}_1,\frac{m}{m_2}\mathbf{s}_2,\ldots,\frac{m}{m_r}\mathbf{s}_r\right)\;\text{(see equation \eqref{eq42})},\; \text{for {\bf Case 2}  in proof of Theorem \ref{th2}.}
			\end{cases}
		\end{align}
	\end{small}
	
	Let $H_{\mathbf{\tilde{s}}}\left(\mathbf{x}\right)$ be the higher $\mathbf{\tilde{s}}$-quasi-homogeneous term of $H\left(\mathbf{x}\right)$. By the process of proving Theorem  \ref{th2}, we know  that $\mathbf{x}=\mathbf{0}$ is a unique zero of $H_{\mathbf{\tilde{s}}}\left(\mathbf{x}\right)$ in $\mathbb{R}^n$.
	
	Let $f_i^\mathbf{\tilde{s}}$ be the higher $\mathbf{\tilde{s}}$-quasi-homogeneous term of $f_i$ for $i=1,\ldots,n$. Then, the higher $\mathbf{\tilde{s}}$-quasi-homogeneous  part of the  polynomial map $F$ is
	$F_\mathbf{\tilde{s}}\left(\bf{x}\right)=\left(f_1^\mathbf{\tilde{s}},\ldots,f_n^\mathbf{\tilde{s}}\right)$. Since
	$$H\left(\bf{x}\right)=\dfrac{\parallel F\left(\bf{x}\right)\parallel_2^2}{2}=\frac{1}{2}\sum_{i=1}^nf_i^2,$$
	one can get that
	\begin{align}\label{eq24}
		&0\leq H_\mathbf{\tilde{s}}\left(\bf{x}\right)\leq\frac{\parallel F_\mathbf{\tilde{s}}\left(\bf{x}\right)\parallel_2^2}{2}.
	\end{align}
	Suppose that $F_\mathbf{\tilde{s}}\left(\bf{a}\right)=\bf{0}$ with $\bf{a}\neq\bf{0}$. By equation \eqref{eq24},  $H_\mathbf{\tilde{s}}\left(\bf{a}\right)=0$ with $\bf{a}\neq\bf{0}$, a contradiction. So,
	$F_\mathbf{\tilde{s}}\left(\bf{x}\right)$ only vanishes at the origin in $\mathbb{R}^n$.
\end{proof}

\begin{proof}[{\bf Proof of Theorem \ref{th11}}]
	The process of proving Theorem  \ref{th2} tells us that Theorem \ref{th2}  can be derived from Theorem \ref{th10}, that is, Theorem \ref{th10}   implies Theorem \ref{th2}.  Next, we will show that Theorem \ref{th10} can be obtained from Theorem \ref{th2}.
	
	Let $H_\mathbf{s}\left(\mathbf{x}\right)$  be the higher $\mathbf{s}$-quasi-homogeneous term of
	$$H\left(\mathbf{x}\right)=\dfrac{\parallel F\left(\mathbf{x}\right)\parallel_2^2}{2}=\dfrac{f_1^2+\cdots+f_n^2}{2}$$
	with  $\mathbf{s}=\left(s_1,\ldots,s_n\right)\in\mathbb{N}_+^n$.  Suppose that $\partial_{\mathbf{x}}H_{\mathbf{s}}\left(\mathbf{x}\right)$
     only vanishes at the origin. Then  $H_\mathbf{s}\left(\mathbf{x}\right)\geq0$ for all $\mathbf{x}\in\mathbb{R}^n$, by Corollary \ref{co2}, $H_{\mathbf{s}}\left(\mathbf{x}\right)$ has the unique zero of $\mathbf{x}=\mathbf{0}$.

We claim that $\partial_{x_i}H_\mathbf{s}\left(\mathbf{x}\right)\not\equiv0$ for $i=1,\ldots,n$. Otherwise,
	$$\mathbf{x}=\left(0,\ldots,0,\overset{i\;\text{th}}{1},0\ldots,0\right)$$
	is  a real zero of $H_\mathbf{s}\left(\mathbf{x}\right)$, a contradiction. Thus, the higher $\mathbf{s}$-quasi-homogeneous part of vector field $\mathscr{Y}$ is $-\partial_{\mathbf{x}}H_\mathbf{s}\left(\mathbf{x}\right)$, which only vanishes at $\mathbf{x}=\mathbf{0}$ in $\mathbb{R}^n$.  By Theorem \ref{th2}, $F$ is injective.
	
So, Theorem \ref{th2}   implies Theorem \ref{th10}.
\end{proof}

\section*{Acknowledgments}
The first author is partially supported by the  National Natural Science Foundation of China (Grant no. 12171491). The second author is partially supported by the National Natural Science Foundation of China (Grant no. 12401205 and Grant no. 12371182), the Guangdong Basic and Applied Basic Research Foundation (Grant no. 2023A1515110430), the National Funded Postdoctoral Researcher Program (Grant no. GZC20230970) and the Fundamental Research Funds for the Central Universities (Grant no. 21624347).

%\bibliographystyle{unsrt}
%\bibliographystyle{elsarticle-num}
%\bibliographystyle{elsarticle-num-names}
%\bibliographystyle{elsarticle-harv}
%\biboptions{square,numbers,sort&compress}
%\bibliographystyle{plain}
%\bibliographystyle{siam}
%\bibliography{ND}

\end{document}